\documentclass[letterpaper, 10 pt, conference]{ieeeconf}  % Comment this line out
\IEEEoverridecommandlockouts                              % This command is only
\newcommand{\yalmip}{\texttt{YALMIP}\;}
\newcommand{\mosek}{\texttt{MOSEK}\;}

\def\R{\mathbb{R}}
\def\dd{\mathrm{d}}
\def\veps{\varepsilon}
\def\vphi{\varphi}

\def\optr{\mathrm{tr}}

\usepackage{graphics} % for pdf, bitmapped graphics files
\usepackage{amsmath} % assumes amsmath package installed
\usepackage{amssymb}  % assumes amsmath package installed
\usepackage{booktabs}

\usepackage{placeins}
\usepackage{color}
\usepackage[hidelinks]{hyperref}

\usepackage{hyphenat}

\newtheorem{lemma}{Lemma}[section]
\newtheorem{remark}[lemma]{Remark}

\newtheorem{proposition}[lemma]{Proposition}

\newtheorem{definition}[lemma]{Definition}

\title{\LARGE \bf
On inverse optimal control
via polynomial optimization
}

\author{J\'er\'emy Rouot and Jean-Bernard Lasserre% <-this % stops a space
\thanks{J. Rouot is with LAAS-CNRS, 7 avenue du colonel Roche, F-31400 Toulouse, France,
        {\tt\small j.rouot@laas.fr}}%
\thanks{Jean B. Lasserre is with LAAS-CNRS and Institute of Mathematics University of Toulouse
LAAS, 7 avenue du Colonel Roche 31077 Toulouse, France.
        {\tt\small lasserre@laas.fr}}%
}

\begin{document}

\maketitle
\thispagestyle{empty}
\pagestyle{empty}

%%%%%%%%%%%%%%%%%%%%%%%%%%%%%%%%%%%%%%%%%%%%%%%%%%%%%%%%%%%%%%%%%%%%%%%%%%%%%%%%
\begin{abstract}
We consider the class of control systems where the differential equation, state and control system
are described by polynomials. Given a set of trajectories and a class of Lagrangians,
we are interested to find a Lagrangian in this class for which these trajectories are optimal.
To model this inverse problem we use a relaxed version of Hamilton-Jacobi-Bellman optimality conditions,
in the continuity of previous work in this vein. Then  we provide a general numerical scheme
based on polynomial optimization and positivity certificates, and illustrate the concepts
on a few academic examples.
\end{abstract}

%%%%%%%%%%%%%%%%%%%%%%%%%%%%%%%%%%%%%%%%%%%%%%%%%%%%%%%%%%%%%%%%%%%%%%%%%%%%%%%%
\section{INTRODUCTION}
Given a system dynamics, control and/or state constraints, and a database $\mathcal{D}$ of feasible state-control trajectories $\{x(t;x_0,t_0),u(t;x_0,t_0)\}$ from several initial states (and different initial times) $(x_0,t_0)$, the 
{\it Inverse Optimal Control Problem} (IOCP) consists of computing a Lagrangian $L$ such that each trajectory of the database $\mathcal{D}$ is an optimal solution of the (direct) optimal control problem (OCP) for the cost functional $\int_{t_0}^TL(x(t),u(t))dt$ to minimize. Even though it is an interesting problem on its own, it is also of primary importance in Humanoid Robotics where one tries to understand human locomotion. %\cite{Laumond2017}.
In fact  whether or not human locomotion obeys some optimization principle is an open issue, discussed
for instance in several contributions in  \cite{Laumond2017}.

As most interesting inverse problems, the (IOCP) is in general an ill-conditioned and ill-posed problem. 
For instance, for a source of ill-posedness, consider an (OCP) with terminal state constraint $x(T)\in X_T$ and with Lagrangian $L$. Let $\phi:X\times [0,T]\to\R$ be its associated optimal value function and let $g:X\to\R$ be a continuously differentiable function  such that $g(x)=0$ for all $x\in X_T$. If the
couple $(x^\star(t;x_0,t_0),u^\star(t;x_0,t_0))$ is an optimal solution of the (OCP) then it is also an optimal solution of the (OCP) with new Lagrangian $L-\langle\nabla g,f\rangle$ (where $f$ is the vector field of the dynamics) and the optimal value function associated with the new (OCP) is $\phi+g$. See e.g. \cite{Pauwels2016} for a discussion.
Therefore in full generality, recovering a Lagrangian that has an intrinsic physical meaning  
might well be searching for a needle in a haystack. This is perhaps why previous works have considered
the (IOCP) with restrictive assumptions on the class of Lagrangian to recover. Some restrictions are also more technical and motivated by simplifying the search process.

In the recent work \cite{Pauwels2016} the authors have proposed 
to consider the (IOCP) in relatively general framework. In this framework an {\it $\veps$-optimality certificate} for the trajectories 
in the database $\mathcal{D}$ translates into natural ``positivity constraints" on the unknown optimal value function $\phi$ and Lagrangian $L$, in a {\it relaxed version} of the Hamilton-Jacobi-Bellman (HJB) equations. Moreover some natural normalizing constraint allows to avoid the source of ill-posedness
alluded to above.
Then the (IOCP) is posed as a {\it polynomial optimization} problem and solved via
a hierarchy of semidefinite relaxations indexed by $d\in\mathbb{N}$,
where the unknown is a couple $(\phi,L)$ of polynomials of potentially large degree $d$;
for more details the reader is referred  to \cite{Pauwels2016}.

In the cost criterion $\min \veps + \gamma\,\Vert L\Vert_1$, the regularizing parameter $\gamma>0$  controls the tradeoff between minimizing $\veps$ and
minimizing the sparsity-inducing $\ell_1$-norm $\Vert L\Vert_1$ of the Lagrangian. (One assumes that
a Lagrangian $L$ with a physical meaning will have a sparse polynomial approximation.) In some examples treated in \cite{Pauwels2016} one clearly sees the importance of $\gamma$ to recover a 
sparse meaningful Lagrangian; the larger is $\gamma$ the sparser is the Lagrangian (but at the cost
of a larger $\veps$, due to a limited computing capacity).

\subsection{Contribution}
 As in \cite{Pauwels2016} the approach in this paper is based on the relaxed-HJB framework to state
an $\veps$-optimality certificate for the trajectories in the database $\mathcal{D}$. Indeed as already mentioned in \cite{Pauwels2016}, HJB-optimality equations is the ``perfect tool" to certify global optimality of a given trajectory (Bellman optimality principle is sometimes called a {\it verification} theorem).
On the other hand, instead of searching for $L$ as a polynomial of potentially large degree as in \cite{Pauwels2016},
we now restrict the search of $L$ to some {\it dictionary} which is a family of polynomials of low degree (for instance convex quadratic). In this respect we
also follow other previous approaches (e.g. \cite{laumont2010,Pauwels2014}) which impose restrictive assumptions on the Lagrangian.
A normalizing condition on $L$ avoids some previously mentioned ill-posedness and in principle no regularizing parameter $\gamma$ is needed (that is, one only minimizes $\veps$ which appears in the $\veps$-optimality certificate). Finally, and again as in \cite{Pauwels2016}, the (IOCP) is solved by a hierarchy of semidefinite relaxations. The quality of the solution improves in the hierarchy but of course at a higher computational cost.

To compare this approach with that of \cite{Pauwels2016} we have considered
four optimal control problems (of modest size): The LQG, the Brockett double integrator, and two minimum-time problems. In all cases we obtain better results as the Lagrangian $L$ is recovered exactly 
at a semidefinite relaxation of relatively low order the hierarchy. In particular no regularizing parameter $\gamma$ is needed. At last but not least, a by-product of the approach is to also provide a means to detect
whether the {\it dictionary} of possible Lagrangian is large enough. Indeed if the resulting optimal value $\veps^\star$
remains relatively {\it high}  after several semidefinite relaxations, one may probably infer that
the dictionary is too small to contain a good Lagrangian.

In section II we define the (IOCP) while the relaxed Hamilton-Jacobi-Bellman optimality
conditions are presented in Section III (with some discussion on the ill-posedness of the problem).
The numerical scheme consisting of a hierarchy of semidefinite relaxations is described in Section IV. The method is illustrated on four academic examples and compared with \cite{Pauwels2016}, before some 
discussion of the relative merits of the method.
 
%%%%%%%%%%%%%%%%%%%%%%%%%%%%%%%%%%%%%%%%%%%%%%%%%%%%%%%%%%%%%%%%%%%%%%%%%%%%%%%%
\section{Setting up the problem}
%%%%%%%%%%%%%%%%%%%%%%%%%%%%%%%%%%%%%%%%%%%%%%%%%%%%%%%%%%%%%%%%%%%%%%%%%%%%%%%%

\subsection{Context}

Let $n$ and $m$ be nonzero integers.
Consider the control system
\begin{equation} \label{eq:dyn}
\dot x(t) = f(x(t),u(t))
\end{equation}
where $f:\mathbb{R}^n\times \mathbb{R}^m$ is a (smooth) vector field, and $u:[t_0,T]\rightarrow U \text{ (compact) } \subset \mathbb{R}^m$ are the controls. A control $u(\cdot)$ is said admissible on $[t_0,T]$ ($T>t_0>0$) if 
the solution $x(\cdot)$ of \eqref{eq:dyn} satisfies $x(t_0)=x_0$ for some $x_0 \in \mathbb{R}^n$ and
\begin{equation} \label{eq:constraintXU}
(x(t),u(t))\in X\times U \; \text{a.e. on } [t_0,T], \quad x(T) \in  X_T
\end{equation}
where $X,X_T$ are compact subsets of $\mathbb{R}^n$. The set of admissible controls is denoted 
as $\mathcal{U}$. Let us denote by $(\tau,x(\cdot),u(\cdot))$ a feasible trajectory for (\ref{eq:dyn})-(\ref{eq:constraintXU}) which starts at time $\tau\in [0,T]$ in state $x(\tau)$.
For an admissible control $u(\cdot)$, we consider only integral costs associated with the corresponding 
trajectory $x(\cdot)$,
and defined by :
\begin{equation} \label{eq:cost}
J(t_0,T,x_0,u(\cdot),L) = \int_{t_0}^T L(x(t),u(t))\, \dd t,  
\end{equation}
where $L:\mathbb{R}^n\times \mathbb{R}^m\rightarrow \mathbb{R}$ is a continuous 
Lagrangian. 

\paragraph{Direct optimal control problem (OCP)}
Given a Lagrangian $L$ and an initial condition $x_0\in \mathbb{R}^n$, the (OCP) consists in 
determining a state-control trajectory $(t_0,x(\cdot),u(\cdot))$ solution of \eqref{eq:dyn} with $x(t_0)=x_0$, satisfying the state and control constraints \eqref{eq:constraintXU} and minimizing the cost \eqref{eq:cost}. 
We will refer to this problem as $OCP(t_0,x_0,L)$.\\
The value function $J_L^\star$ associated with $OCP(t_0,x_0,L)$ is defined by
\begin{equation} \label{eq:value-fct-fixed}
J^\star_L(t_0,x_0) =  \inf_{u(\cdot) \in \mathcal{U}}  \; J(t_0,T,x_0,u(\cdot),L)
\end{equation}
if the final time $T$ is fixed, and,
\begin{equation} \label{eq:value-fct-nonfixed}
J^\star_L(x_0) =  \inf_{T,u(\cdot) \in \mathcal{U}}  \; J(0,T,x_0,u(\cdot),L)
\end{equation}
if the final time $T$ is free.

\paragraph{Inverse optimal control problem (IOCP)}
Given a set $\mathcal{D}$ of controls and 
their associated feasible trajectories (obtained by observations) satisfying (\ref{eq:constraintXU}), a couple defined by a 
control system \eqref{eq:dyn} and  a class (or dictionary) $\mathcal{L}$ of Lagrangians (each $L\in\mathcal{L}$ defines a cost functional \eqref{eq:cost}),
the (IOCP) consists in finding a Lagrangian $L\in \mathcal{L}$
such that the trajectories of  $\mathcal{D}$ are optimal solutions  of the (OCP)  associated with $L$.\\
More precisely, 
we consider a finite family 
$\mathcal{D}=\{(t_i,x_i(\cdot),u_i(\cdot))\}_{i\in I}$ of admissible trajectories, indexed by
some set $I$, where $t_i\in [0,T]$ and $u(\cdot)\in \mathcal{U}$ for all $i\in I$.
We want to compute $L\in \mathcal{L}$  such that 
$x_i(\cdot)$ (resp. $u_i(\cdot)$) are optimal state-trajectories (resp. control-trajectories) of $OCP(t_i,x_i(t_i),L),\forall i\in I$.

%%%%%%%%%%%%%%%%%%%%%%%%%%%%%%%%%%%%%%%%%%%%%%%%%%%%%%%%%%%%%%%%%%%%%%%%%%%%%%%%
\section{Hamilton-Jacobi-Bellman to certify global-optimality}
%%%%%%%%%%%%%%%%%%%%%%%%%%%%%%%%%%%%%%%%%%%%%%%%%%%%%%%%%%%%%%%%%%%%%%%%%%%%%%%%

\subsection{Hamilton-Jacobi-Bellman}

Hamilton-Jacobi-Bellman (HJB)  equation is the ``ideal" tool to 
provide sufficient (and in a certain sense almost necessary) conditions for a trajectory to be globally-optimal, see \cite{Bertsekas1995,Kalman1963,bardi2008}. In general, the {\it value function} is
nonsmooth and cannot be identified to a classical solution of the HJB
equation. Yet, there are many approaches to characterize the value function from the
HJB equation, among others the concept of viscosity solutions \cite{bardi2008} 
or the concept of proximal subgradient \cite{Vinter2000}.
In what follows, a  {\it relaxation} of the (HJB) equation provides a simple certificate of 
global-optimality for (OCP) of the form $OCP(t_0,x_0,L)$.

\begin{definition}
Let $L:\mathbb{R}^n\times \mathbb{R}^m\rightarrow \mathbb{R}$ and let $\vphi:[0,T]\times \mathbb{R}^n \rightarrow \mathbb{R}$ be continuously differentiable.
Given a vector field $f$, we define the linear operator $\mathcal{H}_f$ which acts on $(L,\vphi)$ as
\[
\mathcal{H}_f(L,\vphi) = L + \frac{\partial \vphi}{\partial t} + \nabla_x\vphi^\intercal f 
\]   
\end{definition}

The following certificate was at the core of the method proposed in \cite{Pauwels2016}.
\begin{proposition} \label{prop:HJB-sufficient}
Let $\bar u(\cdot)\in\mathcal{U}$ be an admissible control associated with a trajectory $\bar x(\cdot)$ starting from $x_0$ at time $t=t_0$.
Assume there exists a function $\vphi:[t_0,T]\times X\to\mathbb{R}$, continuously differentiable and
such that 
\begin{align} 
\mathcal{H}_f(L,\vphi)(t,x,u)\ge 0,\; \forall (t,x,u)\in [t_0,T]\times X\times U,\label{eq1-HJB} \\
\vphi(T,x) = 0,\; \forall x\in X_T, \label{eq2-HJB} \\
\mathcal{H}_f(L,\vphi)(t,\bar x(t),\bar u(t)) = 0,\, \forall t\in [t_0,T].\label{eq3-HJB}
\end{align}
Then $\bar u(\cdot)$ is an optimal control associated with the trajectory $\bar x(\cdot)$ 
for the problem $OCP(t_0,x_0,L)$.
\end{proposition}
\begin{proof}
For any admissible control $u(\cdot)$ associated with a trajectory $x(\cdot)$, integrating \eqref{eq1-HJB} 
yields
$\int_{t_0}^T L(x(s),u(s))\dd s + \int_{t_0}^T \dd \vphi\,(s,x(s))\ge 0$. Using \eqref{eq2-HJB},
we obtain $ \vphi(t_0,x_0)\le J(t_0,T,x_0,u(\cdot),L)+\vphi(T,x(T))=J(t_0,T,x_0,u(\cdot),L)$. The same applied 
to $\bar{x}(\cdot)$ and $\bar{u}(\cdot)$ yields 
$\vphi(t_0,x_0)=J(t_0,T,x_0,\bar u(\cdot),L)$ and therefore $\bar u(\cdot)$ is optimal for $OCP(t_0,x_0,L)$. 
\end{proof}

\subsection{$\veps$-optimal solution for the inverse problem}

Next, we describe  slightly weaker conditions than those of Proposition \ref{prop:HJB-sufficient}. They are the basis of our numerical method.

\begin{proposition} \label{prop-relaxHJB}
Let $\bar u(\cdot)$ be an admissible control associated with a trajectory $\bar x(\cdot)$ starting from
$x_0$ at $t=t_0$.
Suppose that there exist a real $\veps$ and a continuous function $L:X\times U\rightarrow \mathbb{R}$ 
and continuous differentiable function $\vphi:[0,T]\times X\to\mathbb{R}$ such that 
\begin{align}
&\mathcal{H}_f(L,\vphi)(t,x,u)\ge 0,\; \forall (t,x,u)\in [t_0,T]\times X\times U, \label{cond1:hjb-relax}\\
&-\veps \le \vphi(T,x) \le 0,\; \forall x\in X_T,\label{cond2:hjb-relax}\\
&\int_{t_0}^T \mathcal{H}_f(L,\vphi)(s,\bar x(s),\bar u(s))\dd s \le \veps.\label{cond3:hjb-relax}
\end{align}
Then $\bar u(\cdot)$ is $2\veps$-optimal for $OCP(t_0,x_0,L)$. 
\end{proposition}
\begin{proof}
Following the proof of Proposition \ref{prop:HJB-sufficient}, we get
$\vphi(t_0,x_0)\le  J(t_0,T,x_0,u(\cdot),L),\,\forall u(\cdot)\in \mathcal{U}$ and from  
condition \eqref{cond3:hjb-relax}, we get 
$\vphi(t_0,x_0)\ge  J(t_0,T,x_0,\bar u(\cdot),L) +\vphi(T,\bar x(T))-\veps$ and therefore
$J(t_0,T,x_0,\bar u(\cdot),L)\le \vphi(t_0,x_0)+2\veps$.
\end{proof}

 In general the value function is nonsmooth, but in case where it is
at least continuous, by the Stone-Weierstrass theorem,  it can be approximated 
by polynomials, uniformly on compact sets, and this approximation
 will provide lower and upper bounds for the value function.
Note that in this case, 
 we can approximate the optimal value as closely as desired.

\subsection{Non well-posedness of the inverse problem}
\label{sec:non-wellposedness}

Proposition \ref{prop-relaxHJB}  provides us with a simple and powerful tool to certify global-suboptimality
of a given set of trajectories, e.g. those given in the database $\mathcal{D}$.
However as we have already mentioned, there are many couples $(\vphi,L)$
that can satisfy (\ref{cond1:hjb-relax})-(\ref{cond3:hjb-relax}), and which make the (IOCP) ill-posed in general.  For instance:

%To overcome this ill-posedness problem, one way chosen in \cite{Pauwels2016}) is to minimize some appropriate ``cost" on $(\vphi,L)$ in searching for a good Lagrangian $L$.  
\begin{itemize}
\item {\it Trivial solution.} The zero Lagrangian and the zero function satisfy 
 conditions \eqref{cond1:hjb-relax}-\eqref{cond2:hjb-relax}-\eqref{cond3:hjb-relax} with
 $\veps=0$, independently of the input trajectories. 
\item {\it Conserved quantity.} Let $u(\cdot)$ be an optimal control associated with a trajectory $x(\cdot)$. Suppose that the system dynamics is such that all feasible trajectories satisfy
 an equation of the form $g(x(t),u(t))=0,\, \forall t\in [0,T]$ for some continuous function $g$.
 Then $L=g^2$ is a candidate for the inverse problem associated with the zero value function $\vphi$. \\ 
Such class of optimal control problems is well spread in applications, especially nonholonomic
systems \cite{Jean2014,Jean2017}. The latter falls into sub-Riemannian geometry framework and
have many applications namely in robotics \cite{Jean2013,Laumond1998,gauthier2013}.
 They can be formulated as  
\begin{equation}
\begin{aligned}
\label{pb:SR-ocp}
&L(x) = (\dot x \cdot \dot x)^{1/2},\quad (T \text{ is fixed})\\
&f(x,u) = \sum_{i=1}^m u_iF_i(x),\; x\in X\subset \mathbb{R}^n,\, u\in U\subset \mathbb{R}^m\\
&X_T = \{x_T\},\, x_T\in \mathbb{R}^n
\end{aligned}
\end{equation}
where $(F_1,\ldots\,F_m)$ are $m$ (smooth) orthonormal vector fields.
The Lagrangian
$L(x)$ stands for the length of the admissible curve $x(\cdot)$ .\\
It is well known \cite{Jean2014} that this is a  length minimization problem equivalent to the 
energy minimization problem where the Lagrangian is $L(x,u) = \dot x \cdot \dot x$ 
and also equivalent to the time minimization problem with $L(x,u)=1$ and where the control domain is restricted to: $\lVert u \rVert_2^2=1$.
\item {\it Total variations.} Consider a continuous function $\psi : [0,T]\times X\rightarrow \mathbb{R}$ such that $\psi$ vanishes on $X_T$. Then $(\vphi,L)$ with $\vphi=-\psi$ and $L=-\nabla_x\psi\cdot f$ 
is solution to the (IOCP) regardless of the database $\mathcal{D}$!
\item {\it Conic convexity.} Let $(L,\vphi)$ be a solution of the inverse problem $(IOCP)$ and
let $\psi:[0,T] \times X\rightarrow \mathbb{R}$ be such that
$\psi(T,x)=0, \, \forall x\in X_T$. If  $\tilde \vphi = \vphi+\psi$ then the pair $(L + (\nabla_x \psi)\cdot f),\tilde \phi)$  satisfies the same conditions 
\eqref{cond1:hjb-relax}-\eqref{cond2:hjb-relax}-\eqref{cond3:hjb-relax} with the same $\veps$  and is solution of the same inverse problem.
 \end{itemize}
To illustrate this last issue, we consider the following problem.
Let $\vphi_L$ be a solution of HJB equations associated with the following optimal control problem
\begin{equation*} \label{gen-ocp}
\begin{aligned}
&\underset{T>0, \lVert u\rVert_2\le 1}{\min }\; \; \int_0^T L(x(t))\, \dd t, \quad \text{ (T is free) }\\
(P)_L\qquad \qquad   & \dot x = u,  \\
& x(0) \in X=B_2,\quad x(T) \in  X_T=\partial X
\end{aligned}
\end{equation*}
where $B_2$ denotes the unit ball of $\mathbb{R}^2$.

\begin{lemma}
The optimal control law of $(P)_L$ is given by $u=-\nabla \vphi_L / \lVert \nabla \vphi_L\rVert_2$
and the Lagrangian $L$ satisfies 
$L = \lVert \nabla \vphi_L \rVert_2$.
\end{lemma}

\begin{proposition}  \label{equiv-pb}
Let $p\in \mathbb{N}$ and define $L_p(x) = \lVert x \rVert_2^p$. The value function associated 
with $(P)_{L_p}$ is $J_{L_p}^\star(x) = 1/(p+1)\,(1-\lVert x \rVert_2^{p+1})$. Besides,
for all $p\in\mathbb{N}$, problems $(P)_{L_p}$ have
the same optimal control law given by $u^\star(x) = x/\lVert x\rVert_2$.
\end{proposition}

For $p\in \mathbb{N}$, let $g_p:X \rightarrow \mathbb{R}$ be defined by 
$g_p(x)=\lVert x\rVert_2-1/(p+1)\,(\lVert x\rVert_2^{p+1}+p)$.
Then, $L_p(x) = L_0(x) - \langle \nabla g_p(x), u^\star(x)  \rangle$ and
$J_{L_p}^\star(x) = J_{L_0}^\star(x) + g_p(x)$.

\noindent
\textbf{Normalization.}\\ 
Due to the convexity of the conditions \eqref{cond1:hjb-relax}-\eqref{cond2:hjb-relax}-\eqref{cond3:hjb-relax}, the set of solutions of the inverse problem is a convex cone. 
Hence, to overcome the non uniqueness of solutions of the inverse problem and to avoid
solutions which do not have a physical interpretation, a natural idea is to impose a normalizing constraint (see e.g. \cite{gauthier2013,laumont2010} or more recently \cite{Pauwels2016}). For instance for free terminal time problems, in \cite{Pauwels2016}  one imposes the constraint on $(\vphi,L)$:
\[
\left\vert \,1-\mathcal{A}(L,\vphi)\,\right\vert\,=\,\veps,
\] 
 where $\mathcal{A}$ is the linear operator:
%  \begin{equation} \label{eq:normalization-pauwels}
\[  \mathcal{A}(L,\vphi) = \int_{\tilde X \times U}\! \Big[L+\frac{\partial \vphi}{ \partial t} 
 +\nabla_x \vphi \cdot f\Big]\dd x\dd u.\]
 % \end{equation}
This allows to avoid the trivial solution $(\vphi,L)=(0,0)$. 

 In this paper, we will not use such a normalization. Rather 
 we propose (a) to restrict the search of Lagrangian solutions $L$
 to a class of functions $\mathcal{L}$ introduced in the next section, and (b) 
include a normalization constraint on $L$ which is more appropriate for our numerical scheme.

 %%%%%%%%%%%%%%%%%%%%%%%%%%%%%%%%%%%%%%%%%%%%%%%%%%%%%%%%%%%%%%%%%%%%%%%%%%%%%%%%
\section{Numerical implementation}
\label{sec:sos-approach}
%%%%%%%%%%%%%%%%%%%%%%%%%%%%%%%%%%%%%%%%%%%%%%%%%%%%%%%%%%%%%%%%%%%%%%%%%%%%%%%%

\subsection{Polynomial and semi-algebraic context}

Without any further assumptions on the data $X,X_T,U$ and $f$, 
the ``positivity constraints"  \eqref{cond1:hjb-relax}-\eqref{cond2:hjb-relax}-\eqref{cond3:hjb-relax} 
cannot be implemented efficiently. However if
$X,X_T,U$ are compact basic semi-algebraic sets and $f$ is a polynomial, then
we may and will use powerful positivity certificate from real algebraic 
geometry, e.g. Putinar's Positivstellensatz \cite{Putinar1993}. In doing so, the positivity constraints
 \eqref{cond1:hjb-relax}-\eqref{cond2:hjb-relax}-\eqref{cond3:hjb-relax} translate into linear matrix inequalities (LMI). The maximum size of the latter depends 
 on the degree allowed for the {\it sums-of-squares} (SOS)
 polynomials involved in the positivity certificate. Of course the higher is the degree, the more powerful is the certificate.

In the sequel  the class of Lagrangian $\mathcal{L}$, denoted by $\mathcal{L}_{a,b}$,
is a set of polynomials.  For a multi-index $\alpha=(\alpha_1,\ldots,\alpha_p)\in \mathbb{N}^p, \, |\alpha|=\sum_{i=1}^p\alpha_i$ and for 
$y=(y_1,\ldots,y_p)$, the notation $y^\alpha$ stands for the monomial $y_1^{\alpha_1}\ldots y_p^{\alpha_p}$. 
%The canonical basis of $\mathbb{R}_d[z],\, z\in \mathbb{R}^p$ is $\{x^\alpha\}_{\alpha\in\{0,\ldots,d\}^p}$.
For two vectors $x=(x_1,\ldots,x_n)\in\mathbb{R}^n$, $u=(u_1,\ldots,u_m)\in\mathbb{R}^m$,
and two integers $a,b$,
let $\mathcal{L}_{a,b}\subset \mathbb{R}[x,u]$ be defined by:
\begin{equation}\label{eq:classPab}
\begin{aligned} 
\mathcal{L}&_{a,b}=\Bigg \{L\in \mathbb{R}[x,u] \mid  L(x,u) = m_a(x)^\intercal\, C_a^x\, m_a(x)  \\
&+ m_b(u)^\intercal\, C_{b}^u\, m_b(u),\; C_a^x\in \mathbb{S}^j_+,\,j=\binom{n+a}{a},\\
&C_b^u\in \mathbb{S}^k_+,\,k=\binom{m+b}{b} 
\Bigg\},
\end{aligned}
\end{equation}
where $m_d(z)$ denotes the vector of monomials $(z^\alpha)_{|\alpha| \le d}$
which forms a canonical basis of $\mathbb{R}_d[z]$ and $\mathbb{S}_+^d$ is the set
of real symmetric and positive semidefinite $d\times d$ matrices.
\begin{remark}
A polynomial of $\mathcal{L}_{a,b}$ is the sum of a polynomial 
 $L_x \in \mathbb{R}_{2a}[x]$ and a polynomial $L_u\in \mathbb{R}_{2b}[u]$ such 
 that $L_x$ and $L_u$ are SOS. Conversely, every SOS polynomial $q$
 of $\mathbb{R}_{2d}[x]$ can be written as $q = m_d(x)^\intercal\,  Q\, m_d(x),\, 
 Q\in \mathbb{S}_+$. 
\end{remark}

To avoid the trivial Lagrangian $L=0$, we include the normalizing constraint 
%\eqref{eq:normalization-pauwels} by
 \begin{equation}  \label{eq:constraint-jr}
  \optr(C_a^x) +  \optr(C_b^u) = C,\; C \text{ is a constant}. 
 \end{equation}
where  $\optr(A)$ denotes the trace of a matrix $A$, and $C_a^x,C_b^u$ are the matrices which 
 define $L$ in its definition \eqref{eq:classPab}.
Note that the map $P\in S_+ \mapsto \optr(P) \in \mathbb{R}_+$ is linear and define a norm.

\subsection{Inverse problem formulation}
 \label{sec:iocp-implementation}
For practical computation, the input data consists of 
$N$ trajectories $\mathcal{D} = \{(t_i,x_i(\cdot),u_i(\cdot))\}_{i=1,\ldots, N}$ 
$(t_i,x_i(\cdot),u_i(\cdot))$ starting at time $t_i$ in state $x_i(t_i)$ and admissible for
(\ref{eq:dyn})-(\ref{eq:constraintXU}), for all $i=1,\ldots,N$. In practice, each trajectory
$(t_i,x_i(\cdot),u_i(\cdot))$ is sampled at some points $(kT/s)$, $k=0,1,\ldots s$, of the interval $[0,T]$,
for some fixed integer $s\in\mathbb{N}$. Then with $t_i:=k_iT/s$ for some integer $k_i<s$,
the database $\mathcal{D}$ is the collection
$((k_iT/s,x_i(kT/s),u_i(kT/s))$, $k=k_i,\ldots,s$, $i=1,\ldots,N$, 

Then condition \eqref{cond3:hjb-relax}
is replaced by :
%\[ \displaystyle \frac{1}{s-k_i+1}\sum_{k=k_i}^s \mathcal{H}_f(L,\vphi)(k_i,x_i(kT/s),u_i(kT/s)) \le \veps,\]
%for all $i=1,\ldots,N$. Alternatively one my also consider
%\[ \sum_{i=1}^N\sum_{k=k_i}^s \mathcal{H}_f(L,\vphi)(k_i,x_i(kT/s),u_i(kT/s)) \le M\,\veps,\]
%with $M=\sum_{i=1}^N(s-k_i+1)$.
\[ \displaystyle \sum_{k=k_i}^s \mathcal{H}_f(L,\vphi)(k_i,x_i(kT/s),u_i(kT/s)) \le \veps,\]
for all $i=1,\ldots,N$. Alternatively one my also consider
\[ \sum_{i=1}^N\sum_{k=k_i}^s \mathcal{H}_f(L,\vphi)(k_i,x_i(kT/s),u_i(kT/s)) \le \veps.\]
Then consider the following hierarchy of optimization problems indexed by $p\in\mathbb{N}$:
\begin{equation} \label{pb:IOCP}
\begin{array}{l}
\veps^\star_p=\displaystyle \inf_{L,\vphi,\veps}\! \; \veps \\
\text{s.t. } \; \mathcal{H}_f(L,\vphi)(t,x,u) \ge 0, \, \forall (t,x,u)\in [0,T]\!\times \!X\times \!U,\\
  \qquad \vphi(T,x) \le 0,\, \forall x\in X_T,\\
  \qquad \vphi(T,x_k)\ge -\veps,\,  k=k_i,\ldots,s,\\
 \qquad \displaystyle\sum_{i=1}^N\sum_{k=k_i}^s \mathcal{H}_f(L,\vphi)(k_i,x_i(kT/s),u_i(kT/s)) \le \veps,\\
  \qquad  L \in \mathcal{L}_{a,b}\, \text{ with }\,  \optr(C_a^x) +  \optr(C_b^u) = 1,  \tag{$IOCP_{a,b}$} 
\end{array}
\end{equation}
where  $\vphi \in \mathbb{R}_{2p}[t,x]$ and  $\veps \in \mathbb{R}$.

\begin{remark}
 \eqref{pb:IOCP} is written for control problems with fixed terminal time $T$. For free terminal time problems
 the function $\vphi$ doesn't depend on $t$. 
\end{remark} 
In \eqref{pb:IOCP}, each positivity constraint on $(L,\vphi)$ 
is replaced by a Putinar's positivity 
certificate \cite{Putinar1993}. The latter asserts that if a polynomial $p$ is positive on 
a compact basic semi-algebraic set
$K=\{x\mid g_i(x)\ge 0,\, g_i\in \mathbb{R}[x],\,i=1,\ldots,m \}$
then $p$ can be written as
\begin{equation} \label{eq:putinar}
p(x)=  \sum_{i=0}^m g_i(x)\sigma_i(x),\quad
\forall x\in \mathbb{R}^n,
\end{equation}
for some SOS polynomials $\sigma_i$, $i=0,\ldots,m$ (and where $g_0(x):=1$ for all $x$).
In addition we impose the degree bound ${\rm deg}(\sigma_i\,g_i)\leq 2p$, $i=0,\ldots,m$.

In doing so ($IOCP_{a,b}$) becomes a hierarchy of {\it semidefinite programs} (SDP) indexed
by $p\in\mathbb{N}$ \cite{Lasserre2010,Lasserre2001}. The size of each SDP in the hierarchy
depends on $p$,  i.e., on the strength of Putinar's positivity certificate (\ref{eq:putinar}) used in (\ref{pb:IOCP}). In practice this size is limited by the capabilities of semidefinite solvers.

Our implementation is based on the \yalmip toolbox \cite{yalmip} which provides a SOS interface to 
handle polynomial constraints and then solves the resulting SDP problem by running the \mosek solver \cite{mosek}. We can handle problems up to $6,7$ variables $(t,x,u)$ 
and degree $12$ for $\vphi$. For larger size problems, heuristics and/or techniques exploiting the structure of the problem should be used. Also, using the recent alternative positivity certificates (to (\ref{eq:putinar})) proposed in
\cite{Lasserre2017} and \cite{ahmadi2014} should be helpful.

%%%%%%%%%%%%%%%%%%%%%%%%%%%%%%%%%%%%%%%%%%%%%%%%%%%%%%%%%%%%%%%%%%%%%%%%%%%%%%%%
\section{Illustrative examples}
%%%%%%%%%%%%%%%%%%%%%%%%%%%%%%%%%%%%%%%%%%%%%%%%%%%%%%%%%%%%%%%%%%%%%%%%%%%%%%%%
We have tested our numerical method described in Section \ref{sec:sos-approach} on four examples
also analyzed in \cite{Pauwels2016}.

\subsection{Settings}

The input of our numerical scheme consists of:
\begin{itemize}
\item final time $T$ if the problem has fixed finite horizon,
\item state constraint set $X$ and the control set $U$,
\item state (and final state) constraints $X,X_T$,
\item vector field $f$,
\item class of polynomials $\mathcal{L}_{a,b}$ \eqref{eq:classPab} to which we restrict the search
for the Lagrangian $L$.
\end{itemize}
Since for problems considered here, at each point $(t,x)\in [0,T]\times X$ we know the associated
optimal control $u(t)$, the database $\mathcal{D}$ now consists of the collection
$((t_k,x_k,u(t_k))_{k=1,\ldots,N}\in [0,T]\times X\times U$ 
where for each $k$ the couple $(t_k,x_k)$ is randomly generated in some set  $S_{\mathcal{D}}\subset[0,T]\times X$.

The output of (IOCP) is a triple $(L,\vphi,\veps^\star)$ where $L$ is a polynomial Lagrangian, $\vphi$ is a 
polynomial value function and $\veps^\star$ is 
a parameter which quantify the sub-optimality
 of the solutions associated with $L$. 
 More precisely, according to Proposition \eqref{prop-relaxHJB}, 
  if $J_{L}^\star(t_0,x_0)$ is  the value
 function  of $OCP(t_0,x_0,L)$ then  
  \begin{align*} \label{eq:csq-LeqL0}
  &\vphi(t_k,x_k)\ge J_{L}^\star(t_k,x_k)\! -\! 2\,\veps^\star,\,k=1,\ldots,N,\\
  & \vphi(T,x_k)\ge -\veps^\star,\,k=1,\ldots,N,\\
    &\vphi(t,x)\le J_{L}^\star(t,x), \quad  \vphi(T,x)\le 0, \, \forall (t,x)\in [0,T]\times X.
 \end{align*}
% for all $(t_i,x_i)$ of the database $\mathcal{D}$. 
% \begin{equation} \label{eq:csq-LeqL0}
%  J_{L}^\star (t_i,x_i)-2\veps \le \vphi(t_i,x_i) \le J_{L}^\star (t_i,x_i)-\veps
% \end{equation}
% for all $(t_i,x_i)$ of the database $\mathcal{D}$. 

\subsection{Benchmark direct problems}
We present the optimal control problems that we considered
for our numerical computation. In all cases we know the value function $J_ {\bar L}^\star$ 
associated with the Lagrangian $\bar L$.

\subsubsection{Linear Quadratic case}
\label{pb:LQR}
\begin{align*}
&T=1,\; X=B_2,\; X_T=\mathbb{R}^2,\; U=\mathbb{R}^2,
f=(x_2,u)^\intercal.
\end{align*}
where $B_2$ denotes the unit ball in $\mathbb{R}^2$. This is the celebrated LQ-problem
and here the Lagrangian to recover is 
$\bar L = 2x_1^2+1/2\,x_1x_2+x_2^2+u^2$. \\
The optimal control law is given by the closed-loop 
control $u(t) = I_2\,B\,E(t)x(t)$ where, for this problem $B=(1\quad 0)$,
 $E(.)$ is solution of the corresponding Riccati equation ($E(T)=0$)
and the value function is given by $J_{\bar L}^\star(t,x(t)) = -x(t)^\intercal E(t)x(t)$.

\subsubsection{Minimum exit-norm in dimension \texorpdfstring{$2$}{ex1}}
\label{pb:exitnorm}
\begin{align*}
T \mbox{ is free},\; X=U=B_2,\; X_T=\partial X,\; f=u.
\end{align*}
The Lagrangian to recover is $\bar L=\lVert x \rVert_2^2+\lVert u \rVert_2^2$.
The associated direct OCP is called the minimum exit-norm problem because at $T^\star$, $x(T^\star)\in\partial X$. 
The optimal control law is $u=x$ and the value function is the 
polynomial $J_{\bar L}^\star(x)=1-\lVert x \rVert_2^2$.

\subsubsection{Minimum exit-time in dimension \texorpdfstring{$2$}{ex2}}
\label{pb:exittime}
\begin{align*}
T \mbox{ is free},\; X=U=B_2,\; X_T=\partial X,\; f=u.
\end{align*}
The Lagrangian to recover is $\bar L=1$.
The optimal control law is $u=x/\lVert x \rVert_2$ and the value function is $J_{\bar L}^\star(x)=1-\lVert x \rVert_2$. It corresponds to the problem $(P)_{L_0}$ introduced in Section \ref{sec:non-wellposedness}.

\subsubsection{Minimum time Brockett integrator}
\label{pb:heisenberg}
\begin{align*}
&T \mbox{ is free},\; X=3B_2,\, U=B_2,\; X_T=\{0\},\\
 &f=(u_1,u_2,x_2u_1-x_1u_2)^\intercal.
\end{align*}
The Lagrangian to recover is $\bar L=1$.
The optimal control law and the value function are described in \cite{trelat2005}.

\subsection{Numerical results}

According to our optimization inverse problem \eqref{pb:IOCP},
we consider several
classes of polynomials $\mathcal{L}_{a,b}$ in which we seek a Lagrangian $L$ solution
of (IOCP). We vary also the degree of the test function $\vphi$.
Results of the problems \ref{pb:LQR}, \ref{pb:exitnorm}, \ref{pb:exittime}, \ref{pb:heisenberg} are
respectively given in Table \ref{table:LQR}, Table \ref{table:exitnorm}, Table \ref{table:exittime} 
and Table \ref{table:heisenberg}.\\
We generate $N=500$ data points from a set $\mathcal{D}\subset X$ and our algorithm implements the normalization constraint \eqref{eq:constraint-jr}.

\begin{table}[!ht]
\setlength\tabcolsep{3pt}
\def\arraystretch{1.2}
\begin{tabular}{cc|ll} 
 $\deg \vphi$ & class of $L$ & $\veps^\star$ & L \\
 4 &  $\mathcal{L}_{1,1}$ & $7e-2$ & $0.78x_1^2+0.82x_1x_2+2.11x_2^2+1.12u^2$ \\
  10 &  $\mathcal{L}_{1,0}$ & $3.1e-1$ & $2.67x_1^2-2.31x_1x_2+1.33x_2^2$ \\
 10 &  $\mathcal{L}_{1,1}$ & $4.5e-6$ & $2x_1^2+0.5x_1x_2+x_2^2+u^2$ \\
  10 &  $\mathcal{L}_{2,2}$ & $4.5e-6$ & $2x_1^2+0.5x_1x_2+x_2^2+u^2$ \\
\end{tabular}
\caption{Solution $(L,\veps^\star)$ of the problem 
$(IOCP)$ associated with the problem \ref{pb:LQR}.\label{table:LQR} }
\end{table}

\begin{table}[!ht]
\setlength\tabcolsep{3.5pt}
\def\arraystretch{1.2}
\begin{tabular}{cc|ll} 
$\deg \vphi$ & class of $L$ & $\veps^\star$ & L \\
 2 & $\mathcal{L}_{1,1}$ & $0$ & $x_1^2+x_2^2+u_1^2+u_2^2$ \\
  2 & $\mathcal{L}_{2,2}$ & $0$ & $x_1^2+x_2^2+u_1^2+u_2^2$ \\
    4 & $\mathcal{L}_{1,1}$ & $0$ & $x_1^2+x_2^2+u_1^2+u_2^2$ \\
  4 & $\mathcal{L}_{2,2}$ & $0$ & $x_1^2+x_2^2+u_1^2+u_2^2$ \\
    2 & $\mathcal{L}_{0,1}$ & $2e-3$ & $1.97+0.54(u_1^2+u_2^2)$ \\
%  4 & $\mathcal{L}_{2,2}$ & $0$ & $0.1(u_1^2+u_2^2+x_1^2+x_2^2)$\\
%    & & & $\;\; + 0.2(x_1^2+x_2^2)^2+0.67(u_1^2+u_2^2)^2$ \\
%      10 & $\mathcal{L}_{2,2}$ & $0$ & $0.156(x_1^2+u_1^2)+0.155(x_2^2+u_2^2)$\\
%    & & & $\;\; + 0.2(x_1^2+x_2^2)^2+0.68(u_1^2+u_2^2)^2$ \\
\end{tabular}
\caption{Solution $(L,\veps^\star)$ of the problem 
$(IOCP)$ associated with the problem \ref{pb:exitnorm}.\label{table:exitnorm}}
\end{table}

\begin{table}[!ht]
\def\arraystretch{1.2}
\setlength\tabcolsep{2.2pt}
\begin{tabular}{ccc|ll} 
$S_\mathcal{D}$ & $\deg \vphi$ & class of $L$ & $\veps^\star$ & L \\
$B_2$ & 4 & $\mathcal{L}_{0,1}$ & $1e-1$ & $0.31+0.34u_1^2+0.36u_2^2$ \\
$B_2$ &  12 & $\mathcal{L}_{0,1}$ & $2e-2$ & $0.327+0.335u_1^2+0.337u_2^2$ \\
$B_2\! \setminus \! 1/2B_2$ & 12 & $\mathcal{L}_{0,1}$ & $2e-4$ & $0.338+0.326u_1^2+0.336u_2^2$ \\
$B_2$ & 2 & $\mathcal{L}_{1,1}$ & $4.5e-2$ & $0.337u_1^2+0.339u_2^2$\\
& & & & $+0.741x_1^2+0.738x_2^2$ \\ %phi = 1/2(1-x1^2-x2^2)
$B_2$ & 12 & $\mathcal{L}_{1,1}$ & $3e-4$ & $x_1^2+x_2^2$\\
$B_2$ & 4 & $\mathcal{L}_{0,2}$ & $0$ & $(1-u_1^2-u_2^2)^2$\\
$B_2$ & 4 & $\mathcal{L}_{2,2}$ & $0$ & $(1-u_1^2-u_2^2)^2$
\end{tabular}
\caption{Solution $(L,\veps^\star)$ of the problem 
$(IOCP)$ associated with the problem \ref{pb:exittime}. $S_\mathcal{D}\subset X$ is the sample from which $\mathcal{D}$ is generated. \label{table:exittime}}
\end{table}

\begin{table}[!ht]
\setlength\tabcolsep{4pt}
\def\arraystretch{1.2}
\begin{tabular}{cc|ll} 
   $\deg \vphi$ & class of $L$ & $\veps^\star$ & L \\
 10 &  $\mathcal{L}_{0,1}$ & $8.31e-2$ & $0.313+0.339\,u_1^2+0.348\,u_2^2$ \\
%  10 &  $\mathcal{L}_{1,0}$ & $?$ & $?$ \\
 14 &  $\mathcal{L}_{0,1}$ & $4.36e-2$ & $0.323+0.338\,u_1^2+0.339\,u_2^2$ \\
% 10 &  $\mathcal{L}_{1,1}$ & $0$ & $?$ \\
 10 &  $\mathcal{L}_{0,2}$ & $0$ & $(1-u_1^2-u_2^2)^2$ \\
 10 &  $\mathcal{L}_{2,2}$ & $0$ & $(1-u_1^2-u_2^2)^2$ \\
  12 &  $\mathcal{L}_{1,1}$ & $1e-1$ & $m_1(x)^\intercal C_1^xm_1(x)+0.31u_1^2+$\\
 & & & $0.35+0.33u_2^2$, $\lVert C_1^x\rVert=O(1e-2)$\\
\end{tabular}
\caption{Solution $(L,\veps^\star)$ of the problem 
$(IOCP)$ associated with the problem \ref{pb:heisenberg}. \label{table:heisenberg}}
\end{table}

\subsection{Discussion}

Clearly, two Lagrangians are equivalent up to a multiplicative constant. 
This multiplicative constant is set by the value of the constant $C$ in 
\eqref{eq:constraint-jr}.
\begin{itemize}
\item Problem \ref{pb:LQR}: Table \ref{table:LQR} gives the couple
$(\veps^\star,L)$ part of the solution of the inverse problem (IOCP) 
for several values of the degree of the
polynomial $\vphi$ and several classes of polynomials for the Lagrangian $L$.
The value function $J_{\bar L}^\star$ associated with the Lagrangian $\bar L$ given in Problem \ref{pb:LQR}
 is quadratic in $x$, but it depends also on the time $t$ and we don't know, \textit{a priori}, if the dependance with
respect to $t$ of $J_{\bar L}^\star(t,x)$ is polynomial. This provides a reason why $L$ is not a good
approximation of $\bar L$ when $\vphi$ is of degree $4$ in $t,x$. 
Contrary to the method presented in \cite{Pauwels2014},
increasing the degree of $\vphi$ up to 10, we are able to get exactly $L=\bar L$ with $\veps^\star=0$, 
provided that the monomial basis for $L$ contains the monomials $u^2,x_1x_2,x_1^2,x_2^2$. 
If we remove some of these monomials, $\veps^\star$ is bigger. Thus it provides
a way to test whether the class $\mathcal{L}_{a,b}$ chosen for $L$ is relevant. The smaller is 
the optimal value $\veps^\star$ the more relevant is $\mathcal{L}_{a,b}$.

\item Problem \ref{pb:exitnorm}: The Lagrangian $\bar L(x,u)$ is 
polynomial in $x,u$ and the value function $J_{\bar L}(x)$
is a polynomial of degree $2$ in $x$.  Numerically, from Table \ref{table:exitnorm} 
taking $\vphi$ as a polynomial of
degree $2$ in $x$, the Lagrangian $L$ solution of (IOCP) corresponds exactly to $\bar L$
provided that the basis of $L$ contains the monomials $u_i^2,x_i^2,\,i=1,2$. As soon as
one of these monomials is removed, the error $\veps^\star$ increases, indicating
that the (smaller) class $\mathcal{L}_{a,b}$ for the Lagrangian $L$ is not relevant. 
\item 
Problem \ref{pb:exittime}: The control
associated with an optimal trajectory satisfies the algebraic equation $u_1^2+u_2^2=1$.
The Lagrangian $L=a\,(1+u_1^2+u_2^2),\, a\neq 0$ is equivalent to 
$\bar L=1$. 
From Table \ref{table:exittime}, if the basis of the Lagrangian $L$ contains the monomials
$1,u_i,u_i^2,\, i=1,2$, we  recover the minimum time problem \ref{pb:exittime}
with a reasonable value of $\veps^\star$ which decreases if the degree of $\vphi$ increases. Note
also that due to the singularity of the value function $J_{\bar L}^\star=1-\lVert x \rVert_2$ at $x=0$, 
$J_{\bar L}^\star$
is hard to approximate by a polynomial near $(0,0)$, hence the objective value $\veps^\star$ of (IOCP) 
with sample from  $S_{\mathcal{D}} = B_2\setminus (1/2\, B_2)$ is smaller.\\
Due to the fact that the optimal control $u=(u_1,u_2)$ satisfies $1-u_1^2-u_2^2=0$,
a natural Lagrangian is $L=(1-u_1^2-u_2^2)^2$ associated with the zero value function and
 obtained by our numerical procedure when the dictionary $\mathcal{L}_{a,b}$ contains 
 the monomials $1,u_i^2,u_i^4,\, i=1,2$. This ``natural" Lagrangian is purely mathematical with no physical meaning.

Note that Problem \ref{pb:exittime} corresponds to $(P)_{L_0},\, L_0=1=\bar L$ which is defined 
in Section \ref{sec:non-wellposedness} and  has the same optimal control law than
 problems $(P)_{L_p},\, p\in \mathbb{N}$.\\
Consider the problems $(P)_{L_1}$ and $(P)_{L_2}$. 
Note that, in these two problems, one function among $L_{p}$ and $J_{L_p}^\star,\,p=1,2$ is a 
polynomial while the other is not differentiable at $0$.
\begin{itemize}
\item $(P)_{L_1}$: The Lagrangian to recover is $L_1(x) = \lVert x \rVert_2$
associated with the value function $J_{L_1}^\star(x) = 1/2\, (1-\lVert x \rVert_2^2)$.\\
From Table \ref{table:exittime}, in the case $\deg \vphi=2$ and $L\in \mathcal{L}_{1,1}$,
the numerical solution $(\vphi,L)$ of the (IOCP) is $\vphi(x) = 1/2\,(1-x_1^2-x_2^2)$ and
$L(x,u)=0.337u_1^2+0.339u_2^2+0.741x_1^2+0.738x_2^2$  that we  identify to 
$L(x,u)=0.34+0.74\,(x_1^2+x_2^2)$ since $u_1^2+u_2^2=1$. 
$L$ corresponds to an approximation
on $X=B_2$ of the Lagrangian $L_1$.
Indeed, using \texttt{MATLAB}'s routine {\it fminunc}, a numerical solution of  
\begin{align*}
\hspace{-0.4cm} \min_{a\in \mathbb{R}^6} \int_X &\mid a_1 + a_2x_1
+a_3x_2+a_4x_1^2+a_5x_1x_2+a_6x_2^2\\
& - \sqrt{x_1^2+x_2^2} \mid \,\dd x_1 \dd x_2 
\end{align*}
is $a^\star=(0.317,0,0,0.7321,0,0.7321)$ whose values  are close to the coefficients of $L$.
 
\item $(P)_{L_2}$: The Lagrangian to recover is $L_2(x) = \lVert x \rVert_2^2$
associated with the value function $J_{L_2}^\star(x) = 1/3\,(1-\lVert x \rVert_2^3)$.\\
From Table \ref{table:exittime}, in the case $\deg \vphi=12$ and $L\in \mathcal{L}_{1,1}$, 
the numerical solution $L$ of the (IOCP) is $L(x,u)=x_1^2+x_2^2$.
In this case, $L$ corresponds to  the Lagrangian $L_2$ 
associated with the  value function $J_{L_2}^\star$.
\end{itemize}

%If the basis of $L$ contains the monomials $1,u_i^2,x_i^2,\,i=1,2$ then we find the Lagrangian 
%$L=1/4 + 1/2(x_1^2+x_2^2)$ which is not surprising.
%Indeed, the problem
%\begin{equation}
%\begin{aligned} \label{pb:equiv-norm}
%&\underset{T,\lVert u \rVert_2\le 1}{\min}\; \tilde L_0(x) = \sqrt{x_1^2+x_2^2} \\
%& \quad \qquad 
%\dot x=u,\, x(0)\in B_2,\,x(T)\in \partial B_2
%\end{aligned}
%\end{equation}
%has the optimal control law $u=x/\lVert x \rVert_2$ which is also the optimal control law of Problem 
%\ref{pb:exittime}. Moreover the value function associated with
%\eqref{pb:equiv-norm} is a polynomial in $x$ defined by $J_{\tilde L_0}(x) = 1/2-1/2\,\lVert x \rVert_2^2$.  \\

\item 
Problem \ref{pb:heisenberg}: The optimal control satisfies $1-u_1^2-u_2^2=0$. From Table \ref{table:heisenberg},
if the basis contains of $\mathcal{L}_{a,b}$ the monomials $1,u_i^2,\, i=1,2$, we recover the Lagrangian of 
the minimum time since the Lagrangian $L=1+u_1^2+u_2^2$ is equivalent
to the Lagrangian $2\bar L$.
If we remove some monomials among $1,u_i^2,\, i=1,2$, the value of $\veps^\star$
increases, which again invalidates the choice of the smaller class $\mathcal{L}_{a,b}$ for $L$. 

If the basis contains the monomials $1,u_i^2,u_i^4,\, i=1,2$ we recover
the Lagrangian $L=(1-u_1^2-u_2^2)^2$  associated with the zero value function
(see Section \ref{sec:non-wellposedness}).\\

It is important to point out that we are able to recover the Lagrangian $\bar L$ without 
adding  a regularization parameter which was necessary in \cite{Pauwels2016}. 
Indeed in \cite{Pauwels2016} the authors adopted a more general point of view where 
$\mathcal{L}=\mathbb{R}[x,u]_{d}$ with $d$ relatively large.
They minimize a tradeoff between the error $\veps^\star$  and the value of a regularization
parameter controlling the {\it sparsity} of the polynomial $L$: the more one asks for sparsity, the
larger  is the resulting optimal error $\veps^\star$ and vice versa. 
From an applicative point of view, the main drawback of this method 
 is that we do not know an {\it a priori} value of the regularization parameter
for which the Lagrangian $\bar L$ is recovered. In this sense,
the method described in this paper -- where a particular a class of Lagrangian functions is imposed --
seems more suitable for applications. 

%As a general qualitative remark, it is not surprising to recover a Lagrangian is SOS.
%Indeed, the map $S_+ \rightarrow \mathbb{R},\; P \mapsto \optr (P)$ is a norm 

\end{itemize}

\section{Conclusion}
Hamilton-Jacobi-Bellman conditions  provide
a global-optimality certificate which is a powerful and ideal tool for solving inverse optimal control problems.
While intractable in full generality, it can be implemented 
when data of the control system are polynomials and semi-algebraic sets, in the spirit of \cite{Pauwels2016}.
Indeed powerful positivity certificates from real algebraic geometry allow to implement a relaxed version of HJB
which imposes (i) positivity constraints on the unknown value function and Lagrangian, and (ii) a guaranteed $\veps$-global optimality for all the given trajectories of 
the database $\mathcal{D}$.
In doing so we can solve efficiently inverse optimal control problems of
relatively small dimension.

Compared to \cite{Pauwels2014}, our method is less general as the search of an optimal Lagrangian is done on some (restricted) class of Lagrangians defined {\it a priori}. On the other hand, the search is more efficient with {\it no} need of a regularizing parameter to control the sparsity. 
When considering the same examples as in \cite{Pauwels2016}, we obtain more accurate estimations on Lagrangians without a sparsity constraint on the Lagrangian. 
Finally, as an additional and interesting feature of the method, in a way a small resulting optimal value $\veps^\star$ validates the choice of the  class of Lagrangians considered, while a relatively large $\veps^\star$ is a strong indication that the class is too small.

Future investigations will try to determine if our numerical scheme provides 
interesting results on practical inverse problems, especially those coming 
from humanoid robotics, an important  field of application \cite{laumont2010} for inverse optimal control.

%%%%%%%%%%%%%%%%%%%%%%%%%%%%%%%%%%%%%%%%%%%%%%%%%%%%%%%%%%%%%%%%%%%%%%%%%%%%%%%%
\section{ACKNOWLEDGMENTS}

This work was supported by the European Research Council (ERC)
through an ERC-Advanced Grant for the TAMING project.
The authors would like to thank Edouard Pauwels for     
 helpful discussions based on upstream works.

%%%%%%%%%%%%%%%%%%%%%%%%%%%%%%%%%%%%%%%%%%%%%%%%%%%%%%%%%%%%%%%%%%%%%%%%%%%%%%%%
\FloatBarrier

\end{document}